\documentclass[11pt]{article}

\usepackage{amsmath,amssymb}
\usepackage{amstext,mathrsfs,mathtools}
\usepackage{amsthm}
\usepackage{amsfonts}
\usepackage{bbm}

\usepackage{ascmac}

\usepackage[dvipdfmx]{graphicx}
\usepackage{hyperref}

\usepackage{authblk}
\usepackage{ulem}

\usepackage{cite}

\usepackage[usenames]{color}

\def\R{\mathbb{R}}
\def\C{\mathbb{C}}

\def\P{{\rm P}}

\def\F{\mathcal{F}}

\def\1{\mathbbm{1}}

\def\lam{\pmb{\lambda}}
\def\T{T_{{\rm col}}}
\def\TT{T_{{\rm col},0}}
\def\l{\ell}

\theoremstyle{definition}

\newtheorem{theorem}{Theorem}[section]
\newtheorem{lemma}[theorem]{Lemma}
\newtheorem{prop}[theorem]{Proposition}
\newtheorem{cor}[theorem]{Corollary}

\newtheorem{ex}[theorem]{Example}

\newtheorem{rem}[theorem]{Remark}

\makeatletter
 \renewcommand{\theequation}{%
 \thesection.\arabic{equation}}
 \@addtoreset{equation}{section}
\makeatother

\title{Eigenvalue processes of symmetric tridiagonal matrix-valued processes associated with Gaussian beta ensemble}
\author[ ]{Satoshi Yabuoku}
\affil[ ]{\textit{Department of Creative Engineering}}
\affil[ ]{\textit{National Institute of Technology, Kitakyushu College}}
\affil[ ]{\textit{5-20-1 Shii, Kokuraminamiku, Kitakyushu, Fukuoka, 802-0985 Japan
}}
\affil[ ]{\textit{yabuoku@kct.ac.jp}}
\date{\today}

\begin{document}
\allowdisplaybreaks
\maketitle

\begin{abstract}
We consider the symmetric tridiagonal matrix-valued process associated with Gaussian beta ensemble (G$\beta$E) by putting independent Brownian motions and Bessel processes on the diagonal entries and upper (lower)-diagonal ones, respectively. Then, we derive the stochastic differential equations that the eigenvalue processes satisfy, and we show that eigenvalues of their (indexed) principal minor sub-matrices appear in the stochastic differential equations. By the Cauchy's interlacing argument for eigenvalues, we can characterize the sufficient condition that the eigenvalue processes never collide with each other almost surely, by the dimensions of the Bessel processes.
\end{abstract}

\section{Introduction}\ 
In random matrix theory, Dyson \cite{Dyson} introduced symmetric, hermitian and symplectic matrix-valued processes having matrix entries as independent Ornstein-Uhlenbeck processes, and he derived the stochastic differential equations (SDEs for short) that their eigenvalue processes satisfy:
\begin{align}\label{eq:O-Udyson}
d\lambda_i(t)=\sqrt{\frac{2}{\beta}}dB_i(t)+\sum_{j(\neq i)}\frac{1}{\lambda_i(t)-\lambda_j(t)}dt-\frac{\lambda_i(t)}{2}dt,\ \ i=1,\dots,N,
\end{align}
where $B_i, i=1,\dots,N$ are independent one-dimensional standard Brownian motions, and the three matrix symmetries correspond to $\beta=1,2$ and $4$, respectively. The invariant measure of the diffusion process $\lam=(\lambda_1,\dots,\lambda_N)$ satisfying \eqref{eq:O-Udyson} is 
\begin{align}\label{eq:densityofGbetaE}
\mu_{\beta,N}(dx)=\frac{1}{Z_{\beta,N}}e^{-\frac{\beta}{4}\sum_{i=1}^Nx_i^2}\prod_{1\le i<j \le N}|x_i-x_j|^{\beta}\prod_{i=1}^Ndx_i,
\end{align}
and \eqref{eq:densityofGbetaE} is the probability distribution of eigenvalues of GOE, GUE and GSE with $\beta=1,2$ and $4$, respectively \cite{AGZ, Mehta04}. Such a Gibbs measure is called a Coulomb gas because of the logarithmic interactions.
Dyson's results of matrix-valued processes are generalized by Katori and Tanemura \cite{KT04}, and it is proved that when matrix entries are replaced with independent Brownian motions, their eigenvalue processes satisfy the SDEs (up to a time change):
\begin{align}\label{matDyson}
d\lambda_i(t)=\sqrt{2}dB_i(t)+{\beta}\sum_{j(\neq i)}\frac{1}{\lambda_i(t)-\lambda_j(t)}dt,\ \ \beta>0,\ \ i=1,\dots,N.
\end{align}
Here, the three matrix symmetries correspond to $\beta=1,2$ and 4, respectively, also see \cite[Section 4.3]{AGZ}. Dyson \cite{Dyson} also found that the system of independent one-dimensional Brownian motions with the condition that they never collide with each other satisfy \eqref{matDyson}, also see \cite{KT11} and references therein. Such stochastic processes satisfying \eqref{eq:O-Udyson} or \eqref{matDyson} are called Dyson's Brownian motion. In this paper, we refer to Dyson's Brownian motion as a solution of \eqref{matDyson}. The solutions of SDEs \eqref{eq:O-Udyson} and \eqref{matDyson} with general $\beta>0$ are studied. Rogers and Shi \cite{RS93} proved that if $\beta\ge1$, these solutions never collide with each other nor explode almost surely. C\'epa and L\'epingle \cite{CL97} proved the existence of unique strong solution for $\beta>0$ with a reflecting boundary condition. 

Recently, constructions of eigenvalue processes which satisfy SDEs \eqref{matDyson} for general $\beta>0$ are reported. As mentioned above, it is clear for $\beta=1,2$ and 4. Allez and Guionnet \cite{AG13} constructed a sequence of symmetric (hermitian) matrix-valued processes whose eigenvalue processes converge weakly to a stochastic process, and the limiting process satisfies \eqref{matDyson} for $0<\beta\le1$ ($0<\beta \le 2$, respectively) with the Ornstein-Uhlenbeck drift. Holcomb and Paquette consider a symmetric tridiagonal matrix-valued process whose eigenvalue process satisfies \eqref{matDyson} (or \eqref{eq:O-Udyson}) and give a sufficient condition of such matrices in \cite[Theorem 6]{HP17}. Their construction relies on the Lanczos algorithm. Fukushima, Tanida and Yano proved \cite{FTY} that a $2\times 2$ symmetric matrix-valued process, which is defined by putting independent Brownian motions and Bessel processes on diagonal entries and off-diagonal ones, respectively, has the eigenvalue process satisfying \eqref{matDyson} with general $\beta>0$. Note that this symmetric matrix-valued process is a natural time dependent model of Gaussian beta ensemble (G$\beta$E for short), introduced by Dumitriu and Edelman \cite{DE}. G$\beta$E is a well-known model because this is a generalization of GOE, GUE and GSE in the following sense. 
For $\beta>0$, G$\beta$E is an $N \times N $ symmetric tridiagonal matrix $H_{\beta}^{tri}$ defined by
\begin{align}\label{eq:defofGbetaE}
H_{\beta}^{tri}=
\dfrac{1}{\sqrt{\beta}}\begin{pmatrix}
g_1 & \chi_{(N-1)\beta}  & 0   & \cdots & \cdots & 0\\
 \chi_{(N-1)\beta} & g_2 &  \chi_{(N-2)\beta} &  &    & \vdots \\
 0 & \chi_{(N-2)\beta} &    & \ddots & & \vdots\\
 \vdots &  & \ddots&  \ddots&  \chi_{2\beta}& 0\\
 \vdots &   &  & \chi_{2\beta} & g_{N-1} & \chi_{\beta}\\
 0  & \cdots &\cdots  & 0 & \chi_{\beta} & g_N
\end{pmatrix},
\end{align}
where $\{g_i\}_{1\le i \le N}$ are independent Gaussian random variables $\mathcal{N}(0,2)$, and $\{\chi_{k\beta}\}_{1\le k\le N-1}$ are independent $\chi$-distributed random variables with shape parameter $k\beta$ and independent of $\{g_i\}_{1\le i \le N}$. Then, Dumitriu and Edelman \cite{DE} proved that eigenvalue probability distribution of $H_{\beta}^{tri}$ is \eqref{eq:densityofGbetaE} with $\beta>0$. For related topics of G$\beta$E, see \cite[Chapter 20]{ABF11} and \cite[Section 4.5]{AGZ}. 

In this paper, we introduce the $N\times N$ symmetric tridiagonal matrix-valued process $H_{\alpha}(t)$ defined by \eqref{defHt}, associated with G$\beta$E. More precisely, $H_{\alpha}(t)$ is constructed by putting independent Brownian motions and Bessel processes with dimension $\alpha_k,k=1,\dots,N-1$ on diagonal entries and off-diagonal ones, respectively. Note that $H_{\alpha}(t)$ does not meet the sufficient condition that their eigenvalue process satisfies \eqref{matDyson} as reported in \cite{HP17}. In our main result (Theorem \ref{thm1}), we derive the SDEs \eqref{sde1} of eigenvalue process of $H_{\alpha}(t)$ for general $N >0$, and we conclude that several eigenvalues of indexed principal minor sub-matrices of $H_{\alpha}(t)$ appear in the SDEs. Hence, the SDEs that eigenvalue processes of $H_{\alpha}(t)$ satisfy are non-Markov type. Formally, drift terms of the obtained SDEs tell us that the eigenvalues of $H_{\alpha}(t)$ interact with their minor eigenvalues, and these interactions are similar to that of Dyson's Brownian motion \eqref{matDyson}, see Remark \ref{rem1}. Moreover, we characterize the sufficient condition that eigenvalue processes of $H_{\alpha}(t)$ never collide with each other almost surely, by using the dimensions $\alpha_k,k=1,\dots,N-1$ (and starting points) of Bessel processes. As a corollary, for $2\times 2$ (sub-)matrices, we recover the construction of eigenvalue processes done in \cite{FTY}. We remark some related topics for minor eigenvalue processes. In the each case of symmetric, hermitian and symplectic-matrix valued process, namely Dyson's Brownian motion with $\beta=1,2$ and 4 respectively, the eigenvalue process with their minor eigenvalues has been studied in \cite{ANP14}. Adler, Nordenstam and van Moerbeke proved that eigenvalue processes of two consecutive minors for these matrices are diffusion processes, and they derive the SDEs that these processes satisfy. Our SDEs \eqref{sde1} slightly look like their SDEs in \cite[(2.27)]{ANP14}, however, ours are more complicated. They also showed that in the case of three consecutive minors, eigenvalue processes with minor eigenvalues are non-Markov. On the other hand, in the context of non-intersecting diffusion processes, multilevel versions of Dyson's Brownian motions are studied  \cite{GS15} and known as corner process. Note that the eigenvalue processes with their minors and corner processes are very different processes as mentioned in \cite{ANP14}.

The organization of the paper is the following. Preliminaries and our main results are in Section \ref{sec2}. We give proofs of Theorem \ref{thm1} and Corollary \ref{cor:DysoneqofN=2} in Section \ref{sec3}. Appendix provides some linear algebraic tools.

\section{Preliminaries and Main Results}\label{sec2}
Let $\alpha_k, k=1,\dots,N-1$ be fixed positive constants. Then for $\alpha=(\alpha_1,\dots,\alpha_{N-1})$, define the $N \times N$ symmetric tridiagonal matrix-valued process $H_{\alpha}(t)$ as following:
\begin{align}\label{defHt}
H_{\alpha}(t):=
\begin{pmatrix}
\sqrt{2}B_1(t)  & X_{1}(t)  & 0 &  \cdots & \cdots & 0\\
 X_{1}(t) & \sqrt{2}B_2(t) &  X_{2}(t)  &  &  & \vdots \\
 0 & X_{2}(t) & & \ddots & & \vdots\\
 \vdots &  & \ddots& \ddots&  X_{N-2}(t)& 0\\
 \vdots &   &   & X_{N-2}(t) & \sqrt{2}B_{N-1}(t) & X_{N-1}(t)\\
 0  & \cdots &\cdots  & 0 & X_{N-1}(t) & \sqrt{2}B_N(t)
\end{pmatrix},\ \ t\ge0,
\end{align}
where $B_k, k=1,\dots,N$ are independent one-dimensional standard Brownian motions defined on a filtered probability space $(\Omega,\F,\{\F_t\}_{t\ge0},\P)$, and $X_k, k=1,\dots,N-1$ are independent Bessel processes started at $x_k\ge0$ with dimension $\alpha_k$, denoted by ${\rm BES}^{\alpha_k}(x_k)$, which are defined on the same probability space and independent of $\{B_k\}_{1\le k \le N}$. Equivalently, $X_k, k=1,\dots,N-1$ satisfy the following SDEs:
\begin{align*}
\begin{dcases}
& dX_k(t)=dB_{kk+1}(t)+\dfrac{\alpha_k-1}{2}\frac{1}{X_k(t)}dt\\
& X_k(0)=x_k
\end{dcases},
\end{align*}
where $B_{kk+1}, k=1,\dots,N$ are independent one-dimensional standard Brownian motions which are independent of  $\{B_k\}_{1\le k \le N}$. 
Recall that for $x\ge0$, $\alpha >0$, ${\rm BES}^{\alpha}(x)$ is a diffusion process on $[0,\infty)$ which is the solution of
\begin{equation}\label{BES}
\begin{dcases}
& dX(t)=dB(t)+\dfrac{\alpha-1}{2}\frac{1}{X(t)}dt,\ t <T_0^{x,\alpha}\\
& X(0)=x
\end{dcases},
\end{equation}
where $B$ is a one-dimensional standard Brownian motion, and $T_0^{x,\alpha}:=\inf\{t>0; X(t)=0\}$ is the first hitting time of $X$ at the origin. Note that \eqref{BES} is well-defined until $T_0^{x,\alpha}$ and if $x>0, \alpha \ge 2$, the point 0 is polar, that is,
\begin{align}\label{BESp}
T_0^{x,\alpha}=\infty,\ \ a.s., 
\end{align}
so that $X(t), t\in[0,\infty)$ lies in the positive half line $(0, \infty)$ almost surely. On the other hand, if $0<\alpha<2$, $\P(T_0^{x,\alpha}<\infty)=1$ and $X$ is recurrent, see \cite[Chapter X\hspace{-1pt}I]{RY}.  Hence, to make $H_{\alpha}$(t) well-defined, \eqref{defHt} is defined up to time $T_0:=\min_{1\le k \le N-1}T_0^{x_k,\alpha_k}$ when there exists $k$ such that $0<\alpha_k<2$. Note that $H_{\alpha}(t)$ can be seen as a time evolution of G$\beta$E because if we take $\alpha=((N-1)\beta, (N-2)\beta, \dots, 2\beta, \beta)$ for some $\beta>0$, by \eqref{eq:defofGbetaE} we have $\frac{1}{\sqrt{\beta}}(H_{\alpha}(1)-H_{\alpha}(0))=H_{\beta}^{tri}$ in distribution.\\
For $1\le p<q \le N$, denote $H_{\alpha}^{p,q}(t):=\{H_{\alpha, k\ell}(t)\}_{p\le k,\ell\le q}$ as the $(q-p+1)$-principal minor sub-matrix of $H_{\alpha}(t)$. By definition of $H_{\alpha}(t)$, $H_{\alpha}^{p,q}(t), 1\le p<q \le N$ are also symmetric and tridiagonal. Let $\lam(t)=(\lambda_1(t), \dots, \lambda_N(t))$ and $\lam^{p,q}(t)=(\lambda_1^{p,q}(t), \dots, \lambda_{q-p+1}^{p,q}(t)), 1 \le p <q \le N$ be the eigenvalue process of $H_{\alpha}(t)$ and that of $H_{\alpha}^{p,q}(t)$ with increasing order, respectively. Note that $\lam^{1,N}(t)=\lam(t)$. For $1\le p<q\le N$, let $f^{p,q}(\lam^{p,q}(t), \lambda)$ be the characteristic polynomial of $H^{p,q}_{\alpha}(t)$. For $1\le p <q \le N$, define the first collision time of $\lam^{p,q}(t)$ by
\begin{align}\label{colpq}
\T^{p,q}:=\inf\{t >0; \lambda_i^{p,q}(t)=\lambda_j^{p,q}(t)\ {\rm for\ some}\ i \neq j\},
\end{align}
and the first collision time of eigenvalues of $H_{\alpha}(t)$ through the indexed minors by    
\begin{align}\label{col}
\T:=\min_{p,q;\ q-p>1}\T^{p,q}.
\end{align}
We assume the following initial condition: 
\begin{align}\label{inicond}
H_{\alpha}(0)\ {\rm\ has\ simple\ spectrum,\ that\ is,\ } \lambda_1(0)<\dots<\lambda_N(0).
\end{align}
Note that if $x_k>0, k=1,\dots,N-1$, this condition actually holds by the recurrence of determinants of tridiagonal matrices, see Lemma \ref{strsep}. For an $N \times N$ square matrix $A$,\ denote $A_{k|\ell}$ as the $(N-1) \times (N-1)$ minor matrix that is obtained by removing the $k$-th row and the $\ell$-th column from A. Then, we have the main result.

\begin{theorem}\label{thm1}
For fixed positive constants $\alpha=(\alpha_1,\dots,\alpha_{N-1})$, up to time $\TT:=\T\wedge T_0$, the eigenvalue process $\lam(t)=(\lambda_1(t), \dots, \lambda_N(t))$ of $H_{\alpha}(t)$ is a continuous semimartingale and satisfies the SDEs:
\begin{align}\label{sde1}
&d\lambda_i(t)=\sqrt{2}\sum_{k=1}^N\dfrac{f(\lam^{1,k-1}(t), \lambda_i(t))f(\lam^{k+1,N}(t), \lambda_i(t))}{\prod\limits_{j(\neq i)}(\lambda_i(t)-\lambda_j(t))}dB_k(t)\nonumber\\
&+2\sum_{k=1}^{N-1}\dfrac{X_k(t)f(\lam^{1,k-1}(t), \lambda_i(t))f(\lam^{k+2,N}(t), \lambda_i(t))}{\prod\limits_{j(\neq i)}(\lambda_i(t)-\lambda_j(t))}dB_{kk+1}(t)
+2\sum_{j(\neq i)}\dfrac{1}{\lambda_i(t)-\lambda_j(t)}dt\nonumber\\
&+\sum_{k=1}^{N-1}\dfrac{(\alpha_k-2)f(\lam^{1,k-1}(t), \lambda_i(t))f(\lam^{k+2,N}(t), \lambda_i(t))}{\prod\limits_{j(\neq i)}(\lambda_i(t)-\lambda_j(t))}dt\nonumber\\
&+\dfrac{2}{\prod\limits_{j(\neq i)}(\lambda_i(t)-\lambda_j(t))^2}\Biggl\{\biggl(\sum_{j(\neq i)}\dfrac{2}{\lambda_i(t)-\lambda_j(t)}\biggr)\sum_{\ell-k>1}F^{k,\l}(\lam^{1,k-1}(t),\lam^{k+1,N}(t),\lam^{1,\l-1}(t),\lam^{\l+1,N}(t),\lambda_i(t))\nonumber\\
&\hspace{40mm} -\sum_{\ell-k>1}\dfrac{\partial}{\partial \lambda}\left(F^{k,\l}(\lam^{1,k-1}(t),\lam^{k+1,N}(t),\lam^{1,\l-1}(t),\lam^{\l+1,N}(t),\lambda)\right)\left.\right|_{\lambda=\lambda_i(t)}\Biggr\}dt,\nonumber\\
&\hspace{105mm} t<\TT,\ \  1\le i \le N,
\end{align}
where
\begin{align*}
f(\lam^{p,q}(t), \lambda)=f^{p,q}(\lam^{p,q}(t), \lambda):=
\begin{cases}
\prod\limits_{r=1}^{q-p+1}(\lambda-\lambda_r^{p,q}(t))&\ \  1\le p< q \le N\\
\ 1&\ \ (p,q)=(1,0), (N+1,N)
\end{cases}\ \ ,
\end{align*}
and for $\l-k>1$, $F^{k,\l}: \R^{k-1}\times\R^{N-k}\times\R^{\l-1}\times\R^{N-\l}\times\R\to\R$ is a function defined by
\begin{align*}
&F^{k,\l}(\lam^{1,k-1}(t),\lam^{k+1,N}(t),\lam^{1,\l-1}(t),\lam^{\l+1,N}(t),\lambda)\\
&:=f(\lam^{1,k-1}(t), \lambda)f(\lam^{k+1,N}(t), \lambda)(\lam^{1,\l-1}(t), \lambda)f(\lam^{\l+1,N}(t),\lambda)\\
&=\prod_{r=1}^{k-1}(\lambda-\lambda_r^{1,k-1}(t))\prod_{r=1}^{N-k}(\lambda-\lambda_r^{k+1,N}(t))\prod_{r=1}^{\l-1}(\lambda-\lambda_r^{1,\l-1}(t))\prod_{r=1}^{N-\l}(\lambda-\lambda_r^{\l+1,N}(t)).
\end{align*}
The quadratic variations are
\begin{align}\label{quad}
d\langle \lambda_i,\lambda_j\rangle_t
=
\begin{dcases}
\ \ 2\biggl(1-2\dfrac{\sum_{\ell-k>1}F^{k,\l}(\lam^{1,k-1}(t),\lam^{k+1,N}(t),\lam^{1,\l-1}(t),\lam^{\l+1,N}(t),\lambda_i(t))}{\prod\limits_{j(\neq i)}(\lambda_i(t)-\lambda_j(t))^2}\biggr)dt & i=j\\
\ \ \dfrac{-4\sum_{\ell-k>1}\det\left(\lambda_i(t)I_N-H_{\alpha}(t)\right)_{k|\ell}\det\left(\lambda_j(t)I_N-H_{\alpha}(t)\right)_{\ell|k}}{\prod\limits_{p(\neq i)}(\lambda_i(t)-\lambda_p(t))\prod\limits_{q(\neq j)}(\lambda_j(t)-\lambda_q(t))}dt & i\neq j
\end{dcases}.
\end{align}

Moreover, if $\alpha_k\ge 2, x_k>0, k=1,\dots,N-1$, then 
\begin{align}\label{noncol}
\TT=\infty,\ \ a.s.
\end{align}
\end{theorem}

\begin{rem}\label{rem1}
The property of the paths of ${\rm BES}^{\alpha}(x)$ mentioned as \eqref{BESp} is essential to show \eqref{noncol}, see Lemma \ref{strsep}. If $\lambda_i(t)$ and $\lam^{p,q}(t)$ never collide with each other, then we have
\begin{align*}
&\dfrac{\partial}{\partial \lambda}f^{p,q}(\lam^{p,q}(t),\lambda)\left.\right|_{\lambda=\lambda_i(t)}=
f^{p,q}(\lam^{p,q}(t),\lambda_i(t))\sum\limits_{r=1}^{q-p+1}\dfrac{1}{\lambda_i(t)-\lambda_r^{p,q}(t)},
\end{align*}
so that the last term in \eqref{sde1} formally becomes
\begin{align}\label{fmdr}
&\dfrac{2}{\prod\limits_{j(\neq i)}(\lambda_i(t)-\lambda_j(t))^2}\Biggl\{\biggl(\sum_{j(\neq i)}\dfrac{2}{\lambda_i(t)-\lambda_j(t)}\biggr)\sum_{\ell-k>1}F^{k,\l}(\lam^{1,k-1}(t),\lam^{k+1,N}(t),\lam^{1,\l-1}(t),\lam^{\l+1,N}(t),\lambda_i(t))\nonumber\\
&\hspace{40mm} -\sum_{\ell-k>1}\dfrac{\partial}{\partial \lambda}\left(F^{k,\l}(\lam^{1,k-1}(t),\lam^{k+1,N}(t),\lam^{1,\l-1}(t),\lam^{\l+1,N}(t),\lambda)\right)\left.\right|_{\lambda=\lambda_i(t)}\Biggr\}\nonumber\\
&=\dfrac{2}{\prod\limits_{j(\neq i)}(\lambda_i(t)-\lambda_j(t))^2}\sum_{\l-k>1}F^{k,\l}(\lam^{1,k-1}(t),\lam^{k+1,N}(t),\lam^{1,\l-1}(t),\lam^{\l+1,N}(t),\lambda_i(t))\nonumber\\
&\times\Biggl(\sum_{j(\neq i)}\dfrac{2}{\lambda_i(t)-\lambda_j(t)}-\sum_{r=1}^{k-1}\dfrac{1}{\lambda_i(t)-\lambda_r^{1,k-1}(t)}-\sum_{r=1}^{N-k}\dfrac{1}{\lambda_i(t)-\lambda_r^{k+1,N}(t)}\nonumber\\
&\hspace{60mm}-\sum_{r=1}^{\l-1}\dfrac{1}{\lambda_i(t)-\lambda_r^{1,\l-1}(t)}-\sum_{r=1}^{N-\l}\dfrac{1}{\lambda_i(t)-\lambda_r^{\l+1,N}(t)}\Biggr). 
\end{align}
This formal expression \eqref{fmdr} is suitable to show that the eigenvalue process $\lam(t)=(\lambda_1(t), \dots, \lambda_N(t))$ of $H_{\alpha}(t)$ interact with their minor eigenvalues, and these interactions are somewhat similar to that of Dyson's Brownian motions as in \eqref{matDyson}. However, \eqref{fmdr} is not well-defined since $\lam(t)$ may collide with their minor eigenvalues. Indeed, when $\alpha_k\ge2, x_k>0, k=1,\dots,N-1$, any consecutive minor eigenvalues $\lam^{p,q}(t),\ \lam^{p,q+1}(t), 1\le p<q<N-1$ never collide almost surely by \eqref{noncol}, but, for example, $\lam^{p,q-1}(t)$ and $\lam^{p,q+1}(t)$ may collide with each other at finite time.
\end{rem}

From the quadratic variations of eigenvalues \eqref{quad}, we have the following relation between the square of difference product of eigenvalues with respect to $\lambda_i(t)$, $\prod_{j(\neq i)}(\lambda_i(t)-\lambda_j(t))^2$, which is \textit{not} the square of the Vandermonde determinant $\det(\{\lambda_i(t)^{j-1}\}_{1\le i,j \le N})=\prod_{i<j}(\lambda_i(t)-\lambda_j(t))$, and difference products of minor eigenvalues with respect to $\lambda_i(t)$:
\begin{align}\label{iden}
\prod\limits_{j(\neq i)}(\lambda_i(t)-\lambda_j(t))^2
&=\sum_{k=1}^Nf(\lam^{1,k-1}(t), \lambda_i(t))^2f(\lam^{k+1,N}(t), \lambda_i(t))^2\nonumber\\
&+2\sum_{k=1}^{N-1}X_k(t)^2f(\lam^{1,k-1}(t), \lambda_i(t))^2f(\lam^{k+2,N}(t), \lambda_i(t))^2\nonumber\\
&+2\sum_{\ell-k>1}F^{k,\l}(\lam^{1,k-1}(t),\lam^{k+1,N}(t),\lam^{1,\l-1}(t),\lam^{\l+1,N}(t),\lambda_i(t)),\nonumber\\
&\hspace{65mm} t<\TT,\ \ 1\le i \le N.
\end{align}
From the non-negativity of $F^{k,\l}(\lam^{1,k-1}(t),\lam^{k+1,N}(t),\lam^{1,\l-1}(t),\lam^{\l+1,N}(t),\lambda_i(t))$ by \eqref{ineq}, this equality \eqref{iden} shows that the diffusion coefficients in \eqref{sde1}  is bounded up to $\TT$. 

\begin{cor}\label{cor1}
For $1\le i, k \le N$, 
\begin{align}\label{coef}
&\left|\dfrac{f(\lam^{1,k-1}(t), \lambda_i(t))f(\lam^{k+1,N}(t), \lambda_i(t))}{\prod\limits_{j(\neq i)}(\lambda_i(t)-\lambda_j(t))}\right|, \ \left|\dfrac{\sqrt{2}X_k(t)f(\lam^{1,k-1}(t), \lambda_i(t))f(\lam^{k+2,N}(t), \lambda_i(t))}{\prod\limits_{j(\neq i)}(\lambda_i(t)-\lambda_j(t))}\right|\ <1,\nonumber\\
&\hspace{120mm} t< \TT.
\end{align}
\end{cor}

Of cause, by \eqref{iden}, $0 \le \frac{2\sum_{\ell-k>1}F^{k,\l}(\lam^{1,k-1}(t),\lam^{k+1,N}(t),\lam^{1,\l-1}(t),\lam^{\l+1,N}(t),\lambda_i(t))}{\prod\limits_{j(\neq i)}(\lambda_i(t)-\lambda_j(t))^2}\le1,\ 1\le i \le N,\ t< \TT$ holds. We remark that Corollary \ref{cor1} is not trivial. For example, take $N=4, i=2, k=3$ and consider the first term in \eqref{coef}, that is, $\frac{(\lambda_2(t)-\lambda_1^{1,2}(t))(\lambda_2(t)-\lambda_2^{1,2}(t))(\lambda_2(t)-\lambda_1^{4,4}(t))}{(\lambda_2(t)-\lambda_1(t))(\lambda_2(t)-\lambda_3(t))(\lambda_2(t)-\lambda_4(t))}$. Here, $\lambda_1^{4,4}(t)=\lambda_1^{4,4}(0)+\sqrt{2}B_{44}(t)$ where $\lambda_1^{4,4}(0)$ is a deterministic constant determined by $x_k, k=1,\dots, N-1$. Then, by the inclusion principle of minor eigenvalues in \cite[Theorem 4.3.28]{HJ}, we have $\lambda_1(t) \le \lambda_1^{1,2}(t) \le \lambda_3(t),\ \lambda_2(t) \le \lambda_2^{1,2}(t) \le \lambda_4(t)$ and $\lambda_1(t) \le \lambda_1^{4,4}(t) \le \lambda_4(t)$. These inequalities lead $|(\lambda_2(t)-\lambda_1^{1,2}(t))(\lambda_2(t)-\lambda_2^{1,2}(t))(\lambda_2(t)-\lambda_1^{4,4}(t))| \le |\lambda_2(t)-\lambda_4(t)|\times\max\{|\lambda_2(t)-\lambda_1(t)|, |\lambda_2(t)-\lambda_3(t)|\}\times\max\{|\lambda_2(t)-\lambda_1(t)|, |\lambda_2(t)-\lambda_4(t)|\}$. Nevertheless, these estimates are not enough to show that $\left|\frac{(\lambda_2(t)-\lambda_1^{1,2}(t))(\lambda_2(t)-\lambda_2^{1,2}(t))(\lambda_2(t)-\lambda_1^{4,4}(t))}{(\lambda_2(t)-\lambda_1(t))(\lambda_2(t)-\lambda_3(t))(\lambda_2(t)-\lambda_4(t))}\right| <1$. Indeed, the tridiagonality of $H_{\alpha}(t)$ is essential to obtain the equation \eqref{iden}, see subsection \ref{subsec3.2} and Remark \ref{rem3}.

Theorem \ref{thm1} and the above observations show $\langle \lambda_i\rangle_t \le 2t$. Roughly speaking, this fact implies that the speeds of eigenvalue processes of $H_{\alpha}(t)$ are smaller than that of Dyson's Brownian motion as in \eqref{matDyson}, and the difference of their speeds is characterized by using minor eigenvalues:
\begin{align*}
2t-\langle \lambda_i\rangle_t=2\int_0^t\frac{2\sum_{\ell-k>1}F^{k,\l}(\lam^{1,k-1}(s),\lam^{k+1,N}(s),\lam^{1,\l-1}(s),\lam^{\l+1,N}(s),\lambda_i(s))}{\prod\limits_{j(\neq i)}(\lambda_i(t)-\lambda_j(s))^2}dt.
\end{align*}
Note that if $\lambda_i(t)$ gets close to other eigenvalues $\lambda_j(t), j\neq i$, then the numerator vanishes by \eqref{iden}. This is valid for the Cauchy's interlacing arguments of minor eigenvalues.

At least, comparing \eqref{sde1} with \eqref{matDyson}, we conclude that the eigenvalue processes of symmetric tridiagonal matrices are very different from Dyson's Brownian motions. However, only when the size of (principal minor sub-)matrices is two, the eigenvalue process $\lam^{p,p+1}(t)=(\lambda_1^{p,p+1}(t), \lambda_2^{p,p+1}(t))$ satisfies \eqref{matDyson}. Note that this fact is already known in \cite{FTY}. Here, we give another proof of the claim by using Theorem \ref{thm1}. Take $p \in \{1,\dots,N-1\}$ and consider the $2\times 2$ principal minor sub-matrix $H_{\alpha_p}^{p,p+1}(t)=
\begin{pmatrix}
\sqrt{2}B_p(t)  & X_{p}(t)\\
 X_{p}(t) & \sqrt{2}B_{p+1}(t)
\end{pmatrix}$
, where $X_p$ is a Bessel process ${\rm BES}^{\alpha_p}(x_p)$. 

\begin{cor}\label{cor:DysoneqofN=2}
Suppose $x_p>0,\  p=1,\dots, N-1$. Then, $\lam^{p,p+1}(t)=(\lambda_1^{p,p+1}(t), \lambda_2^{p,p+1}(t)), t < \TT$ is a  continuous semimartingale and satisfies the SDEs \eqref{matDyson} with $\beta=\alpha_p>0$:
\begin{eqnarray}\label{eq:DysoneqofN=2}
\begin{aligned}
d\lambda_1^{p,p+1}(t)&=\sqrt{2}dB_1^{p,p+1}(t)+\alpha_p\frac{1}{\lambda_1^{p,p+1}(t)-\lambda_2^{p,p+1}(t)}dt,\\
d\lambda_2^{p,p+1}(t)&=\sqrt{2}dB_2^{p,p+1}(t)+\alpha_p\frac{1}{\lambda_2^{p,p+1}(t)-\lambda_1^{p,p+1}(t)}dt,
\end{aligned}
\hspace{10mm}t<\TT,
\end{eqnarray}
where $-x_p=\lambda_1^{p,p+1}(0)<\lambda_2^{p,p+1}(0)=x_p$, and $B_1^{p,p+1}, B_2^{p,p+1}$ are independent one-dimensional standard Brownian motions.
\end{cor}
The proof of Corollary \ref{cor:DysoneqofN=2} lies in subsection \ref{subsec3.3}.

Since each of principal minor sub-matrices $H_{\alpha}^{p,q}(t),\ 1\le p< q \le N,$ are symmetric and tridiagonal, their eigenvalue processes $\lam^{p,q}(t)=(\lambda_1^{p,q}(t), \dots, \lambda_{q-p+1}^{p,q}(t))$ have the same structure of $\lam^{1,N}(t)=\lam(t)$. More precisely, for $1\le p<q \le N$, up to time $\TT$,\ $\lam^{p,q}(t)$ satisfies the following SDEs that are general forms of \eqref{sde1}:
\begin{align}\label{sde2}
&d\lambda_i^{p,q}(t)=\sqrt{2}\sum_{k=p}^q\dfrac{f(\lam^{p,k-1}(t), \lambda_i^{p,q}(t))f(\lam^{k+1,q}(t), \lambda_i^{p,q}(t))}{\prod\limits_{j(\neq i)}(\lambda_i^{p,q}(t)-\lambda_j^{p,q}(t))}dB_k(t)\nonumber\\
&+2\sum_{k=p}^{q-1}\dfrac{X_k(t)f(\lam^{p,k-1}(t), \lambda_i^{p,q}(t))f(\lam^{k+2,q}(t), \lambda_i^{p,q}(t))}{\prod\limits_{j(\neq i)}(\lambda_i^{p,q}(t)-\lambda_j^{p,q}(t))}dB_{kk+1}(t)
+2\sum_{j(\neq i)}\dfrac{1}{\lambda_i^{p,q}(t)-\lambda_j^{p,q}(t)}dt\nonumber\\
&+\sum_{k=p}^{q-1}\dfrac{(\alpha_k-2)f(\lam^{p,k-1}(t), \lambda_i^{p,q}(t))f(\lam^{k+2,q}(t), \lambda_i^{p,q}(t))}{\prod\limits_{j(\neq i)}(\lambda_i^{p,q}(t)-\lambda_j^{p,q}(t))}dt\nonumber\\
&+\dfrac{2}{\prod\limits_{j(\neq i)}(\lambda_i^{p,q}(t)-\lambda_j^{p,q}(t))^2}\nonumber\\
&\times\Biggl\{\biggl(\sum_{j(\neq i)}\dfrac{2}{\lambda_i^{p,q}(t)-\lambda_j^{p,q}(t)}\biggr)\sum_{\substack{p\le k< \l \le q,\\ \ell-k>1}}F^{k,\l}(\lam^{p,k-1}(t),\lam^{k+1,q}(t),\lam^{p,\l-1}(t),\lam^{\l+1,q}(t),\lambda_i^{p,q}(t))\nonumber\\
&\hspace{28mm} -\sum_{\substack{p\le k< \l \le q,\\ \ell-k>1}}\dfrac{\partial}{\partial \lambda}\left(F^{k,\l}(\lam^{p,k-1}(t),\lam^{k+1,q}(t),\lam^{p,\l-1}(t),\lam^{\l+1,q}(t),\lambda)\right)\left.\right|_{\lambda=\lambda_i^{p,q}(t)}\Biggr\}dt,\nonumber\\
&\hspace{90mm} t<\TT,\ \  1\le i \le q-p+1.
\end{align}

\begin{ex}[eigenvalue process for $N=3$]
Take $N=3$ and consider eigenvalue process $\lam(t)\equiv\lam^{1,3}(t)$ of $H_{\alpha}(t)\equiv H_{\alpha}^{1,3}(t)$ with $\alpha=(\alpha_1,\alpha_2)$. Theorem \ref{thm1} and Corollary \ref{cor:DysoneqofN=2} state that the system of $(\lam^{p,q}(t),\ 1\le p \le q \le 3,\ t<\TT)$ satisfies the following system of SDEs:
\begin{align*}
d\lambda_1^{p,p}(t)&=\sqrt{2}dB_p(t), \ \ p=1,2,3,\\
\lam^{1,2}(t)&=(\lambda_1^{1,2}(t), \lambda_2^{1,2}(t))\ {\rm and}\  \lam^{2,3}(t)=(\lambda_1^{2,3}(t), \lambda_2^{2,3}(t))\ {\rm satisfy}\ \eqref{eq:DysoneqofN=2},\\
d\lambda_i(t)&=\frac{\sqrt{2}}{\prod\limits_{j(\neq i)}(\lambda_i(t)-\lambda_j(t))}
 \Bigl\{(\lambda_i(t)-\lambda_1^{2,3}(t))(\lambda_i(t)-\lambda_2^{2,3}(t))dB_1(t)\\
&+(\lambda_i(t)-\lambda_1^{1,1}(t))(\lambda_i(t)-\lambda_1^{3,3}(t))dB_2(t)
+(\lambda_i(t)-\lambda_1^{1,2}(t))(\lambda_i(t)-\lambda_2^{1,2}(t))dB_3(t)\\
&\hspace{20mm}+\sqrt{2}X_1(t)(\lambda_i(t)-\lambda_1^{3,3}(t))dB_{12}(t)
+\sqrt{2}X_2(t)(\lambda_i(t)-\lambda_1^{1,1}(t))dB_{23}(t)
\Bigr\}\\
&+2\sum_{j(\neq i)}\dfrac{1}{\lambda_i(t)-\lambda_j(t)}dt
+\frac{(a_1-2)(\lambda_i(t)-\lambda_1^{3,3}(t))+(a_2-2)(\lambda_i(t)-\lambda_1^{1,1}(t))}{\prod\limits_{j(\neq i)}(\lambda_i(t)-\lambda_j(t))}dt\\
&+\dfrac{2}{\prod\limits_{j(\neq i)}(\lambda_i(t)-\lambda_j(t))^2}\\
&\times\Biggl\{\biggl(\sum_{j(\neq i)}\dfrac{2}{\lambda_i(t)-\lambda_j(t)}\biggr)(\lambda_i(t)-\lambda_1^{2,3}(t))(\lambda_i(t)-\lambda_2^{2,3}(t))(\lambda_i(t)-\lambda_1^{1,2}(t))(\lambda_i(t)-\lambda_2^{1,2}(t))\\
&\hspace{10mm}-(\lambda_i(t)-\lambda_2^{2,3}(t))(\lambda_i(t)-\lambda_1^{1,2}(t))(\lambda_i(t)-\lambda_2^{1,2}(t))\\
&\hspace{10mm}-(\lambda_i(t)-\lambda_1^{2,3}(t))(\lambda_i(t)-\lambda_1^{1,2}(t))(\lambda_i(t)-\lambda_2^{1,2}(t))\\
&\hspace{10mm}-(\lambda_i(t)-\lambda_1^{2,3}(t))(\lambda_i(t)-\lambda_2^{2,3}(t))(\lambda_i(t)-\lambda_2^{1,2}(t))\\
&\hspace{10mm}-(\lambda_i(t)-\lambda_1^{2,3}(t))(\lambda_i(t)-\lambda_2^{2,3}(t))(\lambda_i(t)-\lambda_1^{1,2}(t))\Biggr\}dt,\ \ t<\TT,\ \  1\le i \le 3.
\end{align*}
\end{ex}


\section{Proofs of main results}\label{sec3}
\subsection{Proof of Theorem \ref{thm1}}\label{subsec3.1}
The proof of Theorem \ref{thm1} is similar to that used in our previous paper \cite{Y}. We first derive the SDEs \eqref{sde1} by the implicit function theorem until the first collision time $t \in [0,\TT)$. And then, we show that \eqref{noncol} whenever $\alpha_k\ge 2, x_k>0, k=1,\dots,N-1$. The detail calculations are summarized in subsection \ref{subsec3.2}. Define the $N \times N$ symmetric tridiagonal deterministic matrix $H=\{H_{k\ell}\}_{1\le k,\ell \le N}$ as
\begin{align}\label{3.1}
H:=
\begin{pmatrix}
\sqrt{2}x_1 & y_{12}  & 0  & \cdots & \cdots & 0\\
y_{12} & \sqrt{2}x_2 &  y_{23}  &  &  & \vdots \\
0 & y_{23} &  &  \ddots & & \vdots\\
\vdots & & \ddots& \ \ddots&  y_{N-2N-1}& 0\\
\vdots &   &   & y_{N-2N-1} & \sqrt{2}x_{N-1} & y_{N-1N}\\
 0  & \cdots &\cdots  & 0 & y_{N-1N} & \sqrt{2}x_N
\end{pmatrix},
\end{align}
and its characteristic polynomial of $H$ as $f(\lambda):=\det(\lambda I_N-H)$. Let $\lambda_1\le\dots\le\lambda_N$ be eigenvalues of $H$, and assume that $H$ has simple spectrum. Then, for $i=1,\dots,N$ we have  
\begin{align*}
f_{\lambda}(\lambda_i)=\prod_{j(\neq i)}(\lambda_i-\lambda_j)\neq0,
\end{align*} 
where $f_{\eta}$ is the partial derivative of $f$ with respect to $\eta$. Hence for $i=1,\dots,N$, we apply the implicit function theorem for $\lambda_i$ and obtain the first and second derivatives of them as following:
\begin{align}\label{derilam}
\dfrac{\partial \lambda_i}{\partial \eta}=-\dfrac{f_\eta(\lambda_i)}{f_\lambda(\lambda_i)},\ \ 
\frac{\partial^2 \lambda_i}{\partial \eta^2}=-\frac{f_{\eta\eta}(\lambda_i)+2f_{\lambda\eta}(\lambda_i)\lambda_{i,\eta}+f_{\lambda\lambda}(\lambda_i)\lambda_{i, \eta}^2}{f_\lambda(\lambda_i)}.
\end{align}
For second derivatives of $\lambda_i$, taking the summation for $\eta=x_k, k=1,\dots, N$ and $y_{kk+1}, k=1,\dots,N-1$, we have
\begin{align}\label{lap1}
\Delta \lambda_i=-\frac{\Delta f(\lambda_i)+2\nabla f_{\lambda}(\lambda_i)\cdot \nabla \lambda_i+f_{\lambda\lambda}(\lambda_i)\nabla \lambda_i \cdot \nabla \lambda_i}{f_\lambda(\lambda_i)},
\end{align}
where for a $C^2$ function $g: \R^{2N-1} \to \R$, we denote the gradient and Laplacian of $g$ by
\begin{align*}
\nabla g :=\left(\frac{\partial g}{\partial x_{1}},\frac{\partial g}{\partial x_{2}},\dots,\frac{\partial g}{\partial x_{N}},\frac{\partial g}{\partial y_{12}},\dots,\frac{\partial g}{\partial y_{N-1 N}}\right),\ \ 
\Delta g:=\sum_{k=1}^N\frac{\partial^2 g}{\partial x_{k}^2}+\sum_{k=1}^{N-1}+\frac{\partial^2 g}{\partial y_{kk+1}^2}.
\end{align*}
By using the tridiagonality of $H$, we can calculate \eqref{lap1} and obtain an explicit form.
Denote $I_N$ as an $N \times N$ identity matrix. Recall that for an $N \times N$ square matrix $A$,\ let $A_{k|\ell}$ be the $(N-1) \times (N-1)$ minor matrix that is obtained by removing the $k$-th row and the $\ell$-th column from A. Similarly, let $A_{k\ell |pq}$ be the $(N-2) \times (N-2)$ minor matrix that is obtained by removing the $k,\ell$-th rows and the $p,q$-th columns from $A$. For $1\le p<q\le N$, let $f^{p,q}(\lambda)$ be the characteristic polynomial of $H^{p,q}$. 

\begin{prop}\label{prop1}
For $i=1,\dots,N$, we have
\begin{align}\label{lap2}
\Delta \lambda_i=&\dfrac{2f_{\lambda\lambda}(\lambda_i)}{f_{\lambda}(\lambda_i)}-\dfrac{2}{f_{\lambda}(\lambda_i)}\sum_{k=1}^{N-1}\det\left((\lambda_i I_N-H)_{kk+1|kk+1}\right)\nonumber\\
&+\frac{4}{f_{\lambda}(\lambda_i)^2}\Biggl(\frac{f_{\lambda\lambda}(\lambda_i)}{f_{\lambda}(\lambda_i)}\sum_{\ell-k>1}f^{1,k-1}(\lambda_i)f^{k+1,N}(\lambda_i)f^{1,\ell-1}(\lambda_i)f^{\ell+1,N}(\lambda_i)\nonumber\\
&-\sum_{\ell-k>1}\dfrac{\partial}{\partial \lambda}\left(f^{1,k-1}(\lambda)f^{k+1,N}(\lambda)f^{1,\ell-1}(\lambda)f^{\ell+1,N}(\lambda)\right)\left.\right|_{\lambda=\lambda_i}\Biggr).
\end{align}
\end{prop}

The proof of Proposition \ref{prop1} is in subsection \ref{subsec3.2}. Now we derive the SDEs \eqref{sde1}. Since $H_{\alpha}(t)$ has simple spectrum up to the time $\T$, we are able to use \eqref{derilam}, \eqref{lap1}, \eqref{deriflam} and \eqref{derifxy} in subsection \ref{subsec3.2}. Applying Ito's formula for $\lambda_i(t), i=1,\dots,N$ with $x_k=B_k(t), k=1,\dots,N, y_{kk+1}=X_k(t), k=1,\dots,N-1$, for $t \in [0,\TT)$ we have 
\begin{align}\label{a1}
d\lambda_i(t)=&\sum_{k=1}^N\dfrac{\partial \lambda_i}{\partial x_k}dB_{k}(t)+\sum_{k=1}^{N-1}\dfrac{\partial \lambda_i}{\partial y_{kk+1}}\Bigl(dB_{kk+1}(t)+\dfrac{\alpha_k-1}{2X_k(t)}dt\Bigr)
+\frac{1}{2}\Delta \lambda_i(t)dt\nonumber\\
=&\sqrt{2}\sum_{k=1}^N\dfrac{\det\left((\lambda_i(t) I_N-H_{\alpha}(t))_{k|k}\right)}{f_{\lambda}(\lambda_i(t))}dB_{k}(t)\nonumber\\
&+2\sum_{k=1}^{N-1}\dfrac{X_k(t)\det\left((\lambda_i(t) I_N-H_{\alpha}(t))_{kk+1|kk+1}\right)}{f_{\lambda}(\lambda_i(t))}dB_{kk+1}(t)\nonumber\\
&+\sum_{k=1}^{N-1}\dfrac{(\alpha_k-1)\det\left((\lambda_i(t) I_N-H_{\alpha}(t))_{kk+1|kk+1}\right)}{f_{\lambda}(\lambda_i(t))}dt+\frac{1}{2}\Delta \lambda_i(t)dt. 
\end{align}
By Proposition \ref {prop1} and Lemma \ref{lemA1}, 
\begin{align}\label{a2}
&\dfrac{({\rm drift\ term\ of\ }d\lambda_i(t))}{dt}\nonumber\\
&=\sum_{k=1}^{N-1}\dfrac{(\alpha_k-1)\det\left((\lambda_i(t) I_N-H_{\alpha}(t))_{kk+1|kk+1}\right)}{f_{\lambda}(\lambda_i(t))}
-\dfrac{1}{f_{\lambda}(\lambda_i(t))}\sum_{k=1}^{N-1}\det\left((\lambda_i(t) I_N-H_{\alpha}(t))_{kk+1|kk+1}\right)\nonumber\\
&+\dfrac{f_{\lambda\lambda}(\lambda_i(t))}{f_{\lambda}(\lambda_i(t))}+\frac{2}{f_{\lambda}(\lambda_i(t))^2}\Biggl(\frac{f_{\lambda\lambda}(\lambda_i(t))}{f_{\lambda}(\lambda_i(t))}\sum_{\ell-k>1}f^{1,k-1}(\lambda_i(t))f^{k+1,N}(\lambda_i(t))f^{1,\ell-1}(\lambda_i(t))f^{\ell+1,N}(\lambda_i(t))\nonumber\\
&\ \ \ \ \ \ \ \ \ \ \ \ \ \ \ \ \ \ \ \ \ \ \ \ \ \ \ \ \ \ \ \ \ \ \ \ \ -\sum_{\ell-k>1}\dfrac{\partial}{\partial \lambda}\left(f^{1,k-1}(\lambda)f^{k+1,N}(\lambda)f^{1,\ell-1}(\lambda)f^{\ell+1,N}(\lambda)\right)\left.\right|_{\lambda=\lambda_i}\Biggr)\nonumber\\
&=2\sum_{j(\neq i)}\dfrac{1}{\lambda_i(t)-\lambda_j(t)}+\sum_{k=1}^{N-1}\dfrac{(\alpha_k-2)\det\left((\lambda_i(t) I_N-H_{\alpha}(t))_{kk+1|kk+1}\right)}{f_{\lambda}(\lambda_i(t))}\nonumber\\
&+\frac{2}{f_{\lambda}(\lambda_i(t))^2}\Biggl(\biggl(2\sum_{j(\neq i)}\dfrac{1}{\lambda_i(t)-\lambda_j(t)}\biggr)\sum_{\ell-k>1}f^{1,k-1}(\lambda_i(t))f^{k+1,N}(\lambda_i(t))f^{1,\ell-1}(\lambda_i(t))f^{\ell+1,N}(\lambda_i(t))\nonumber\\
&\ \ \ \ \ \ \ \ \ \ \ \ \ \ \ \ \ \ \ \ \ \ -\sum_{\ell-k>1}\dfrac{\partial}{\partial \lambda}\left(f^{1,k-1}(\lambda)f^{k+1,N}(\lambda)f^{1,\ell-1}(\lambda)f^{\ell+1,N}(\lambda)\right)\left.\right|_{\lambda=\lambda_i(t)}\Biggr).
\end{align}
To obtain \eqref{sde1}, we note that $(\lambda_i(t) I_N-H_{\alpha}(t))_{k|k}$ and $(\lambda_i(t) I_N-H_{\alpha}(t))_{kk+1|kk+1}$ are block diagonal matrices, that is, $(\lambda_i(t) I_N-H_{\alpha}(t))_{k|k}={\rm diag}(\lambda_i(t) I_{k-1}-H_{\alpha}^{1,k-1}(t), \lambda_i(t) I_{N-k}-H_{\alpha}^{k+1,N}(t))$ and $(\lambda_i(t) I_N-H_{\alpha}(t))_{kk+1|kk+1}={\rm diag}(\lambda_i(t) I_{k-1}-H_{\alpha}^{1,k-1}(t), \lambda_i(t) I_{N-k-1}-H_{\alpha}^{k+2,N}(t))$. Hence, we have
\begin{align*}
\det\left((\lambda_i(t) I_N-H_{\alpha}(t))_{k|k}\right)&=f^{1,k-1}(\lam^{1,k-1}(t), \lambda_i(t))f^{k+1,N}(\lam^{k+1,N}(t), \lambda_i(t)),\\
\det\left((\lambda_i(t) I_N-H_{\alpha}(t))_{kk+1|kk+1}\right)&=f^{1,k-1}(\lam^{1,k-1}(t), \lambda_i(t))f^{k+2,N}(\lam^{k+2,N}(t), \lambda_i(t)).
\end{align*}
Therefore, combining \eqref{a1}, \eqref{a2} and the equation $f_{\lambda}(\lambda_i(t))=\prod_{j(\neq i)}(\lambda_i(t)-\lambda_j(t))$, we obtain the SDEs \eqref{sde1} until  the time $\TT$.

Next, we show \eqref{quad}. By \eqref{a1} and \eqref{derifxy} in subsection \ref{subsec3.2}, 
\begin{align}\label{b1}
d\langle \lambda_i,\lambda_j\rangle_t
&=\dfrac{2}{f_{\lambda}(\lambda_i(t))f_{\lambda}(\lambda_j(t))}
\Biggl(\sum_{k=1}^N\det\left((\lambda_i(t) I_N-H_{\alpha}(t))_{k|k}\right)\det\left((\lambda_j(t) I_N-H_{\alpha}(t))_{k|k}\right)\nonumber\\
&+2\sum_{k=1}^{N-1}\det\left((\lambda_i(t) I_N-H_{\alpha}(t))_{k|k+1}\right)\det\left((\lambda_j(t) I_N-H_{\alpha}(t))_{k|k+1}\right)\biggr)dt.
\end{align}
By symmetry of $H_{\alpha}(t)$ and Lemma \ref{lemA4} (the Cauchy-Binet formula) we have
\begin{align}\label{b2}
\Bigl({\rm two\ summations\ in\ \eqref{b1}}\Bigr)
=\sum_{k=1}^N&\det\left((\lambda_i(t) I_N-H_{\alpha}(t)(\lambda_j(t) I_N-H_{\alpha}(t))_{k|k}\right)\nonumber\\
-2\sum_{\l-k>1}&\det\left((\lambda_i(t) I_N-H_{\alpha}(t))_{k|\l}\right)\det\left((\lambda_j(t) I_N-H_{\alpha}(t))_{\l|k}\right).
\end{align}
For the first summation in \eqref{b2}, since $(\lambda_i(t) I_N-H_{\alpha}(t)(\lambda_j(t) I_N-H_{\alpha}(t)$ has eigenvalues $(\lambda_i(t)-\lambda_p(t))(\lambda_j(t)-\lambda_p(t)),\ p=1,\dots, N$, by Lemma \ref{lemA2} we have
\begin{align}\label{b3}
\sum_{k=1}^N\det\left((\lambda_i(t) I_N-H_{\alpha}(t)(\lambda_j(t) I_N-H_{\alpha}(t))_{k|k}\right)&=\sum_{k=1}^N\prod_{p(\neq k)}(\lambda_i(t)-\lambda_p(t))(\lambda_j(t)-\lambda_p(t))\nonumber\\
&=
\begin{cases}
f_{\lambda}(\lambda_i(t))^2 & i=j\\
0 & i \neq j
\end{cases}.
\end{align}
For the second summation in \eqref{b2}, if $i=j$, by Lemma \ref{lemA5} (the Sylvester's identity) we have
\begin{align*}
\begin{vmatrix}
\det\left((\lambda_i(t) I_N-H_{\alpha}(t))_{k|k}\right) &\det\left((\lambda_i(t) I_N-H_{\alpha}(t))_{k|\l}\right) \\
\det\left((\lambda_i(t) I_N-H_{\alpha}(t))_{\l|k}\right) & \det\left((\lambda_i(t) I_N-H_{\alpha}(t))_{\l|\l}\right)
\end{vmatrix}&=f(\lambda_i(t))\det\left((\lambda_i(t) I_N-H_{\alpha}(t))_{k\l|k\l}\right)\\
&=0,
\end{align*}
which gives
\begin{align}\label{b4}
&\det\left((\lambda_i(t) I_N-H_{\alpha}(t))_{k|\l}\right)\det\left((\lambda_i(t) I_N-H_{\alpha}(t))_{\l|k}\right)\nonumber\\
&=\det\left((\lambda_i(t) I_N-H_{\alpha}(t))_{k|k}\right)\det\left((\lambda_i(t) I_N-H_{\alpha}(t))_{\l|\l}\right)\nonumber\\
&=f^{1,k-1}(\lam^{1,k-1}(t), \lambda_i(t))f^{k+1,N}(\lam^{k+1,N}(t), \lambda_i(t))f^{1,\l-1}(\lam^{1,\l-1}(t), \lambda_i(t))f^{\l+1,N}(\lam^{\l+1,N}(t), \lambda_i(t))\nonumber\\
&=F^{k,\l}(\lam^{1,k-1}(t),\lam^{k+1,N}(t),\lam^{1,\l-1}(t),\lam^{\l+1,N}(t),\lambda_i(t)).
\end{align}
Note that by symmetry of $H_{\alpha}(t)$, $\det\left((\lambda_i(t) I_N-H_{\alpha}(t))_{\l|k}\right)=\det\left((\lambda_i(t) I_N-H_{\alpha}(t))_{k|\l}\right)$. Hence, the following non-trivial inequalities with respect to eigenvalues and their minor eigenvalues hold: for $k,\l$ such that $\l-k>1$,
\begin{align}\label{ineq}
F^{k,\l}(\lam^{1,k-1}(t),\lam^{k+1,N}(t),\lam^{1,\l-1}(t),\lam^{\l+1,N}(t),\lambda_i(t)) \ge 0.
\end{align} 
Combining \eqref{b1}-\eqref{b4}, we obtain \eqref{quad}.

\begin{rem}\label{rem2}
By a similar way as the above proof and calculation, we can generalize our main results by taking continuous semimartingales as matrix entries of $H_{\alpha}(t)$. Note that the drift term in \eqref{sde1} can be described by using \textit{only} eigenvalues of $H_{\alpha}(t)$ and those of minors. This occurs by taking Bessel processes as the upper (lower)-diagonal entries of $H_{\alpha}(t)$. Therefore, the choice of matrix entries of $H_{\alpha}(t)$ as \eqref{defHt} is natural and also reasonable from the view point of random matrix theory and the SDEs that $\lam(t)$ satisfies.
\end{rem}

Finally, we show that $\TT$ is infinity almost surely whenever $\alpha_k\ge 2, x_k>0, k=1,\dots,N-1$. By \eqref{BESp}, for all $k=1,\dots,N-1$ we have $T_0^{x_k,\alpha_k}=\infty,$ a.s., so that $T_0=\infty,$ a.s. Combining this, continuity of eigenvalues and Lemma \ref{strsep}, for all $1\le p<q\le N$ we have the strong separation argument for $\lam^{p,q}(t)=(\lambda_1^{p,q}(t), \dots, \lambda_{q-p+1}^{p,q}(t))$, that is, 
\begin{align*}
\lambda_1^{p,q}(t)<\lambda_2^{p,q}(t)<\dots<\lambda_{q-p+1}^{p,q}(t), \ t\in (0,\infty),\ \ a.s.,
\end{align*}
which gives $\T^{p,q}=\infty,$ a.s., defined in $\eqref{colpq}$. Therefore, we also have $\T=\infty,$ a.s., and the claim holds. The above discussion and next subsection \ref{subsec3.2} complete the proof of Theorem \ref{thm1}.


\subsection{Proof of Proposition \ref{prop1}}\label{subsec3.2}
To show Proposition \ref{prop1}, we first compute derivatives of $f$ with respect to $\lambda, x_k,$ and $ y_{kk+1}$. As a result of the implicit function theorem in \eqref{derilam}, we can compute $\Delta \lambda_i$ as in \eqref{lap1}. Note that in this subsection, we heavily rely on linear algebra and matrix analysis, so that we summarize basic tools in Appendix.
\begin{prop}\label{derif}
For variables $x_k, k=1,\dots,N$ and $y_{kk+1}, k=1,\dots,N-1$ in \eqref{3.1}, we have
\begin{eqnarray}
& & \begin{aligned}\label{deriflam}
&f_{\lambda}(\lambda)=\sum_{k=1}^N\det\left((\lambda I_N-H)_{k|k}\right),\ 
f_{\lambda\lambda}(\lambda)=2\sum_{k<\ell}\det\left((\lambda I_N-H)_{k\ell | k\ell}\right),
\end{aligned}\\
& & \begin{aligned}\label{derifxy}
&f_{x_k}(\lambda)=-\sqrt{2}\det\left((\lambda I_N-H)_{k|k}\right), f_{x_kx_k}(\lambda)=0,\\
&f_{y_{kk+1}}(\lambda)=2\det\left((\lambda I_N-H)_{k|k+1}\right)=-2y_{kk+1}\det\left((\lambda I_N-H)_{kk+1|kk+1}\right),\\
&f_{y_{kk+1}y_{kk+1}}(\lambda)=-2\det\left((\lambda I_N-H)_{kk+1|kk+1}\right).
\end{aligned}
\end{eqnarray}
\end{prop}
Remark that \eqref{deriflam} holds for any $N \times N$ matrix and its characteristic function.

\begin{proof}[Proof of \eqref{deriflam}]
We apply the Fredholm determinant expansion to $f(\lambda)=\det(\lambda I_N-H)$ and obtain 
\begin{align*}
f(\lambda)=\lambda^N+\sum_{k=1}^{N}(-1)^k\lambda^{N-k}\sum_{1\le j_1<\cdots<j_k\le N}\underset{1\le \ell,m\le k}{{\rm det}}(H_{j_\ell j_m}),
\end{align*}
where $\underset{1\le \ell,m\le k}{{\rm det}}(H_{j_\ell j_m})$ is the indexed $k$-th principal minor defined in Lemma \ref{lemA2}. Differentiating both sides with respect to $\lambda$, we have
\begin{align}\label{flam}
f_\lambda(\lambda)&=N\lambda^{N-1}+\sum_{k=1}^{N-1}(-1)^k(N-k)\lambda^{N-k-1}\sum_{1\le j_1<\cdots<j_k\le N}\underset{1\le l,m\le k}{{\rm det}}(H_{j_lj_m}),\\
\label{flamlam}
f_{\lambda \lambda}(\lambda)&=N(N-1)\lambda^{N-2}+\sum_{k=1}^{N-2}(-1)^k(N-k)(N-k-1)\lambda^{N-k-2}\sum_{1\le j_1<\cdots<j_k\le N}\underset{1\le l,m\le k}{{\rm det}}(H_{j_lj_m}).
\end{align}
On the other hand, for each $\det\left(\lambda I_{N-1}-H_{k|k}\right), k=1,\dots,N$, using the Fredholm determinant expansion again and \eqref{flam} we have
\begin{align*} 
\sum_{k=1}^N\det\left((\lambda I_N-H)_{k|k}\right)
&=\sum_{k=1}^N\det\left(\lambda I_{N-1}-H_{k|k}\right)\\  
&=\sum_{k=1}^N\left\{\lambda^{N-1}+\sum_{r=1}^{N-1}(-1)^r\lambda^{N-1-r}\sum_{\substack{1\le j_1<\cdots<j_r\le N \\ j_s (\neq i)}}\underset{1\le l,m\le r}{{\rm det}}(H_{j_lj_m})\right\}\\
&=N\lambda^{N-1}+\sum_{k=1}^{N-1}(-1)^k(N-k)\lambda^{N-k-1}\sum_{1\le j_1<\cdots<j_k\le N}\underset{1\le l,m\le k}{{\rm det}}(H_{j_lj_m})\\
&=f_\lambda(\lambda).
\end{align*}
A similar calculation gives
\begin{align*}
&2\sum_{k<\ell}\det\left((\lambda I_N-H)_{k\ell | k\ell}\right)=2\sum_{k<\ell}\det\left(\lambda I_{N-2}-H_{k\ell | k\ell}\right)\\
&=2\sum_{k<\ell}\left\{\lambda^{N-2}+\sum_{r=1}^{N-2}(-1)^r\lambda^{N-2-r}\sum_{\substack{1\le j_1<\cdots<j_r\le N, \\ j_s(\neq k,\ell)}}\underset{1\le l,m\le r}{{\rm det}}(H_{j_lj_m})\right\}\\
&=N(N-1)\lambda^{N-2}+2\sum_{r=1}^{N-2}(-1)^r\lambda^{N-2-r}\sum_{k<\ell}\sum_{\substack{1\le j_1<\cdots<j_r\le N, \\ j_s(\neq k,\ell)}}\underset{1\le l,m\le r}{{\rm det}}(H_{j_lj_m})\\
&=N(N-1)\lambda^{N-2}+2\sum_{r=1}^{N-2}(-1)^r\lambda^{N-2-r}\frac{(N-r)(N-r-1)}{2}\sum_{1\le j_1<\cdots<j_r\le N}\underset{1\le l,m\le r}{{\rm det}}(H_{j_lj_m}),
\end{align*}
here the last equation holds since  for fixed indices $(j_1,\dots,j_r)$, the number of pairs $(k,\l)$ such that $\underset{1\le l,m\le r}{{\rm det}}((H_{k\l|k\l})_{j_lj_m})=\underset{1\le l,m\le r}{{\rm det}}(H_{j_lj_m})$ is exactly $\binom{N-r}{2}$. Hence, by \eqref{flamlam} we have \eqref{deriflam}. 
\end{proof}

Next, to show \eqref{derifxy}, we need the following two lemmas.

\begin{lemma}\label{lem3.1}
For any $N \times N$ symmetric matrix $A=(a_{kl})_{1\le k,\ell \le N}$ and $k<\l$,
\begin{align*}
\dfrac{\partial \det A}{\partial a_{kk}}=\det\left(A_{k|k}\right),\ \ \dfrac{\partial \det A}{\partial a_{k\l}}=(-1)^{k+\l}2\det\left(A_{k|\l}\right).
\end{align*} 
\end{lemma}

\begin{proof}
We use Lemma \ref{lemA3}. Since the determinants in \eqref{A3} do not have $(k,k)$ and $(k, \l)$ entries of $A$, by differentiating both sides of \eqref{A3} with respect to $a_{kk}$ we immediately have the first equation. Using $a_{k\l}=a_{\l k}$, we have  
\begin{align*}
&\dfrac{\partial \det(A)}{\partial a_{k\l}}
=-2a_{k\ell}\det(A_{k\ell |\ell k})
+\sum_{\substack{q\neq k,\ell\\q<\ell}}(-1)^{k+q-1}a_{\ell q}\det(A_{k\ell | \ell q})
+\sum_{\substack{q\neq k,\ell\\q>\ell}}(-1)^{k+q}a_{\ell q}\det(A_{k\ell | \ell q})\\
&\hspace{20mm}+\sum_{\substack{p\neq k,\ell\\p>k}}(-1)^{\ell+p-1}a_{kp}\det(A_{k\ell |pk})
+\sum_{\substack{p\neq k,\ell\\p<k}}(-1)^{\ell+p}a_{kp}\det(A_{k\ell |pk})\\
&=(-1)^{k+\l}\left\{\sum_{\substack{q\neq k,\ell\\q<\ell}}(-1)^{\l-1+q}a_{\ell q}\det(A_{k\ell | \ell q})+(-1)^{\l-1+k}a_{\l k}\det(A_{k\ell |\ell k})+\sum_{\substack{q\neq k,\ell\\q>\ell}}(-1)^{\l-1+q-1}a_{\ell q}\det(A_{k\ell | \ell q})\right\}\\
&+(-1)^{k+\l}\left\{\sum_{\substack{p\neq k,\ell\\p<k}}(-1)^{k+p}a_{kp}\det(A_{k\ell |pk})+(-1)^{k+\l-1}a_{k \l}\det(A_{k\ell |\ell k})+\sum_{\substack{p\neq k,\ell\\p>k}}(-1)^{k+p-1}a_{kp}\det(A_{k\ell |pk})\right\}\\
&=(-1)^{k+\l}\det\left(A_{k|\l}\right)+(-1)^{k+\l}\det\left(A_{\l|k}\right)\\
&=(-1)^{k+\l}2\det\left(A_{k|\l}\right),
\end{align*}
where the last equation holds by symmetry of $A$.
\end{proof}

\begin{lemma}\label{lem3.2}
Let $A=(a_{kl})_{1\le k,\ell \le N}$ be an $N \times N$ tridiagonal matrix. Assume that there exists $k_0 \in \{2,\dots,N-1\}$ such that $a_{k_0k_0-1}=a_{k_0k_0}=0$. Then $\det A=0$.  
\end{lemma}
\begin{proof}
Since $A$ is tridiagonal, its determinant has the following recurrence relation:
\begin{align}\label{rec}
\det A^{(k+1)}=a_{k+1k+1}\det A^{(k)}-a_{kk-1}a_{k-1k}\det A^{(k-1)},\ \ k=2,\dots, N-1,
\end{align}
where for $k=1,\dots,N$, $A^{(k)}$ is the $k$-th leading (top left) principal minor of $A$, see \cite{HJ}. Since $A^{(k)}, k=1.\dots, N$ are tridiagonal, by the hypothesis we have $\det A^{(k_0)}=0$. Hence, using \eqref{rec} we find $\det A^{(k_0+1)}=0$. Again, using \eqref{rec} inductively we have $\det A^{(k)}=0, k=k_0,\dots,N$. 
\end{proof}

\begin{proof}[Proof of \eqref{derifxy}]
We apply Lemma \ref{lem3.1} for $A=\lambda I_N-H$ and obtain
\begin{align*}
f_{x_k}(\lambda)=-\sqrt{2}\det\left((\lambda I_N-H)_{k|k}\right),\ \ 
f_{y_{kk+1}}(\lambda)=2\det\left((\lambda I_N-H)_{k|k+1}\right),
\end{align*}
which give $f_{x_kx_k}(\lambda)=0$. We claim that 
\begin{align}\label{det1}
\det\left((\lambda I_N-H)_{k|k+1}\right)=-y_{kk+1}\det\left((\lambda I_N-H)_{kk+1|kk+1}\right).
\end{align}
Indeed, it is clear for $k=1, N-1$. For $k\neq 1, N-1$, expanding $\det\left((\lambda I_N-H)_{k|k+1}\right)$ by the $k$-th row we have
\begin{align*}
&\det\left((\lambda I_N-H)_{k|k+1}\right)=
\begin{vmatrix}
\text {\Large{$(\lambda I_N-H)^{(1,k-1)}$}} & -y_{k-1k} & \text {\Huge{0}} \\
& -y_{kk+1} & -y_{k+1k+2}\ \ \ \ \ \ \ \ \ \ \ \ \ \ \ \   \\
\text {\Huge{0}} & &  \text {\Large{$(\lambda I_N-H)^{(k+2,N)}$}}
\end{vmatrix}\\
&=-y_{kk+1} \begin{vmatrix}
\text {\Large{$(\lambda I_N-H)^{(1,k-1)}$}} & \text {\Huge{0}} \\
\text {\Huge{0}} &  \text {\Large{$(\lambda I_N-H)^{(k+2,N)}$}}
\end{vmatrix}\\
&+y_{k+1k+2}\begin{vmatrix}
& & & 0 & \\ 
 & \text {\Large{$(\lambda I_N-H)^{(1,k-1)}$}} & & \vdots & & & &  \text {\Huge{0}} & \\
 & & & 0 & \\
 & & & -y_{k-1k}  &  \\
  0 & \cdots & 0 &  0& 0 &  y_{k+2k+3} & 0 & \cdots & 0 \\
& &  & &  0 &  & & & \\
& \text {\Huge{0}} &  & & \vdots &  & &  \text {\Large{$(\lambda I_N-H)^{(k+3,N)}$}} \\
& & & & 0 &
\end{vmatrix}.
\end{align*}
Here, the block diagonal matrix is equal to $(\lambda I_N-H)_{kk+1|kk+1}$, and the second determinant vanishes by Lemma \ref{lem3.2}. Therefore, \eqref{det1} holds. Using this, we obtain \eqref{derifxy}, and proof of Proposition \ref{derif} is complete.
\end{proof}

Proposition \ref{derif} shows a key lemma to compute $\Delta \lambda_i$.

\begin{lemma}\label{lem3.3}
For $i=1,\dots,N$, we have
\begin{align}\label{key1}
f_{\lambda}(\lambda)^2-\dfrac{1}{2}\nabla f \cdot \nabla f(\lambda)=-f(\lambda)\Delta f(\lambda)+2\sum_{\l-k>1}\det\left((\lambda I_N-H)_{k|k}\right)\det\left((\lambda I_N-H)_{\l|\l}\right).
\end{align}
\end{lemma}

\begin{proof}
By \eqref{derifxy}, we have
\begin{align*}
\nabla f \cdot \nabla f(\lambda)=2\sum_{k=1}^N\det\left((\lambda I_N-H)_{k|k}\right)^2+4\sum_{k=1}^{N-1}\det\left((\lambda I_N-H)_{k|k+1}\right)^2.
\end{align*}
Using this and \eqref{deriflam}, 
\begin{align*}
&f_{\lambda}(\lambda)^2-\dfrac{1}{2}\nabla f \cdot \nabla f(\lambda)
=2\sum_{k<\l}\det\left((\lambda I_N-H)_{k|k}\right)\det\left((\lambda I_N-H)_{\l|\l}\right)-
\sum_{k=1}^{N-1}\det\left((\lambda I_N-H)_{k|k+1}\right)^2\\
&=2\sum_{k=1}^{N-1}
\begin{vmatrix}
\det\left((\lambda I_N-H)_{k|k}\right) &\det\left((\lambda I_N-H)_{k|k+1}\right) \\
\det\left((\lambda I_N-H)_{k+1|k}\right) & \det\left((\lambda I_N-H)_{k+1|k+1}\right)
\end{vmatrix}
+2\sum_{\l-k>1}\det\left((\lambda I_N-H)_{k|k}\right)\det\left((\lambda I_N-H)_{\l|\l}\right).
\end{align*}
The last equation holds by symmetry of $H$. By Lemma \ref{lemA5} (the Sylvester's identity), for $k=1,\dots, N-1$ we have
\begin{align*}
\begin{vmatrix}
\det\left((\lambda I_N-H)_{k|k}\right) &\det\left((\lambda I_N-H)_{k|k+1}\right) \\
\det\left((\lambda I_N-H)_{k+1|k}\right) & \det\left((\lambda I_N-H)_{k+1|k+1}\right)
\end{vmatrix}=f(\lambda)\det\left((\lambda I_N-H)_{kk+1|kk+1}\right),
\end{align*}
and by \eqref{derifxy},
\begin{align}\label{lapf}
\Delta f(\lambda)=-2\sum_{k=1}^{N-1}\det\left((\lambda I_N-H)_{kk+1|kk+1}\right).
\end{align}
Combining the above three equations, \eqref{key1} holds. 
\end{proof}
\begin{rem}\label{rem3}
Applying Lemma \ref{lem3.3} with $\lambda=\lambda_i(t)$, we immediately obtain \eqref{iden}. 
The indices $(k,\l)$ taken in the summation in \eqref{key1} are the same as zero entries of $H$. Hence, having zero matrix entries gives one reason that eigenvalue processes of symmetric tridiagonal matrix-valued processes interact with their minor eigenvalues as in \eqref{sde1}.
\end{rem}

Now we show Proposition \ref{prop1}.

\begin{proof}[Proof of Proposition \ref{prop1}]
We calculate \eqref{lap1}. Differentiating both sides of \eqref{key1} with respect to $\lambda$, we have
\begin{align}\label{eq2}
&2f_{\lambda}(\lambda)f_{\lambda\lambda}(\lambda)-\nabla f_{\lambda} \cdot \nabla f(\lambda)\nonumber\\
&=-f_{\lambda}(\lambda)\Delta f(\lambda)-f(\lambda)\Delta f_{\lambda}(\lambda)+2\sum_{\l-k>1}\dfrac{\partial}{\partial \lambda}\left(\det\left((\lambda I_N-H)_{k|k}\right)\det\left((\lambda I_N-H)_{\l|\l}\right)\right).
\end{align}
Substituting $\lambda=\lambda_i$ in \eqref{key1} and \eqref{eq2}, by \eqref{derilam} and $f(\lambda_i)=0$ we have 
\begin{align*}
\nabla f_{\lambda}(\lambda_i)\cdot \nabla \lambda_i&=-\dfrac{\nabla f_{\lambda} \cdot \nabla f(\lambda_i)}{f_{\lambda}(\lambda_i)}\\
&=-2f_{\lambda\lambda}(\lambda_i)-\Delta f(\lambda_i)+\dfrac{2}{f_{\lambda}(\lambda_i)}\sum_{\l-k>1}\dfrac{\partial}{\partial \lambda}\Bigl(\det\left((\lambda I_N-H)_{k|k}\right)\det\left((\lambda I_N-H)_{\l|\l}\right)\Bigr)\left.\right|_{\lambda=\lambda_i},\\
\nabla \lambda_i \cdot \nabla \lambda_i&=\dfrac{\nabla f \cdot \nabla f(\lambda_i)}{f_{\lambda}(\lambda_i)^2}\\
&=2-\dfrac{4}{f_{\lambda}(\lambda_i)^2}\sum_{\l-k>1}\det\left((\lambda_i I_N-H)_{k|k}\right)\det\left((\lambda_i I_N-H)_{\l|\l}\right).
\end{align*}
Using these equations, by \eqref{lap1} we have
\begin{align*}
\Delta \lambda_i&=-\frac{\Delta f(\lambda_i)+2\nabla f_{\lambda}(\lambda_i)\cdot \nabla \lambda_i+f_{\lambda\lambda}(\lambda_i)\nabla \lambda_i \cdot \nabla \lambda_i}{f_\lambda(\lambda_i)}\\
&=-\dfrac{1}{f_\lambda(\lambda_i)}\Biggl(-\Delta f(\lambda_i)-2f_{\lambda\lambda}(\lambda_i)
+\dfrac{4}{f_{\lambda}(\lambda_i)}\sum_{\l-k>1}\dfrac{\partial}{\partial \lambda}\Bigl(\det\left((\lambda I_N-H)_{k|k}\right)\det\left((\lambda I_N-H)_{\l|\l}\right)\Bigr)\left.\right|_{\lambda=\lambda_i}\\
&\ \ \ \ \ \ \ \ \ \ \ \ \ \ \ \ \ -\dfrac{4f_{\lambda\lambda}(\lambda_i)}{f_{\lambda}(\lambda_i)^2}\sum_{\l-k>1}\det\left((\lambda_i I_N-H)_{k|k}\right)\Biggl).
\end{align*}
Combining \eqref{lapf}, we have \eqref{lap2} and Proposition \ref{prop1} is proved.
\end{proof}


\subsection{Proof of Corollary \ref{cor:DysoneqofN=2}}\label{subsec3.3}
\begin{proof}
Since $H_{\alpha_p}^{p,p+1}(t)$ is symmetric and tridiagonal, $\lam^{p,p+1}$ satisfies \eqref{sde2} with $q=p+1$. Note that there are no indices $(k,\l)$ such that $p\le k<\l \le p+1,\ \l-k>1$, so that the terms with respect to $F^{k,\l}$ and $\frac{\partial}{\partial \lambda}F^{k,\l}$ in summations in \eqref{sde2} completely vanish. Also, by definition $f(\lam^{p,p-1}(t), \lambda)\equiv f(\lam^{p+2,p+1}(t), \lambda)\equiv1$. Hence, we have
\begin{align*}
d\lambda_1^{p,p+1}(t)&=dM_1^{p,p+1}(t)
+2\dfrac{1}{\lambda_1^{p,p+1}(t)-\lambda_2^{p,p+1}(t)}dt
+\dfrac{\alpha_p-2}{\lambda_1^{p,p+1}(t)-\lambda_2^{p,p+1}(t)}dt\\
&=dM_1^{p,p+1}(t)+\alpha_p\dfrac{1}{\lambda_1^{p,p+1}(t)-\lambda_2^{p,p+1}(t)}dt,\\
d\lambda_2^{p,p+1}(t)&=dM_2^{p,p+1}(t)+\alpha_p\dfrac{1}{\lambda_2^{p,p+1}(t)-\lambda_1^{p,p+1}(t)}dt,\ \ t<\TT,
\end{align*}
where $M_1^{p,p+1}$ and $M_2^{p,p+1}$ are local martingales. Again, since there are no indices $(k,\l)$ such that $p\le k<\l \le p+1,\ \l-k>1$, by \eqref{quad} in Theorem \ref{thm1} we have $d\langle \lambda_i,\lambda_j\rangle_t=2t\delta_{ij}$. Therefore, by L\'evy's characterization for local martingales \cite[Theorem 3.3.16]{KS}, there exist independent one-dimensional standard Brownian motions $B_1^{p,p+1}(t),\ B_2^{p,p+1}(t)$ such that $M_1^{p,p+1}(t)=\sqrt{2}B_1^{p,p+1}(t),\ M_2^{p,p+1}(t)=\sqrt{2}B_2^{p,p+1}(t)$, and we obtain \eqref{eq:DysoneqofN=2}. For the initial conditions, since 
$H_{\alpha_p}^{p,p+1}(0)
=\begin{pmatrix}
0 & x_p\\
x_p & 0
\end{pmatrix}
$, we have $\lambda_1^{p,p+1}(0)=-x_p, \lambda_2^{p,p+1}(0)=x_p$, and $x_p>0$ gives $\lambda_1^{p,p+1}(0)<\lambda_2^{p,p+1}(0)$, which actually satisfies \eqref{inicond}.
\end{proof}


\appendix

\renewcommand{\theequation}{A.\arabic{equation}}
\renewcommand{\thetheorem}{A.\arabic{theorem}}
\def\theschapter{Appendix\Alph{chapter}}
\section{Appendix}

\begin{lemma}[Drift terms of Dyson's model and characteristic polynomial]\label{lemA1}\ \\
Assume that an $N\times N$ matrix $A$ has simple spectrum $\lambda_1,\dots,\lambda_N$, and the characteristic polynomial is $f(\lambda):=\det(\lambda I_N-A)$. Then for $i=1,\dots,N$,
\begin{align}\label{A1}
\frac{f_{\lambda \lambda}(\lambda_i)}{f_\lambda(\lambda_i)}=2\sum_{j(\neq i)}\frac{1}{\lambda_i-\lambda_j}.
\end{align}
\end{lemma}
\begin{proof}
By definition of eigenvalues, 
$f(\lambda)=\prod_{j=1}^N(\lambda-\lambda_j)=(\lambda-\lambda_i)\prod_{j (\neq i)}(\lambda-\lambda_j)$. Differentiating this with respect to $\lambda$, we have
\begin{align*}
f_\lambda(\lambda)&=\prod_{j (\neq i)}(\lambda-\lambda_j)+(\lambda-\lambda_i)\sum_{j (\neq i)}\prod_{k (\neq i,j)}(\lambda-\lambda_k),\\
f_{\lambda \lambda}(\lambda)&=\sum_{j (\neq i)}\prod_{k (\neq i,j)}(\lambda-\lambda_k)+\sum_{j (\neq i)}\prod_{k (\neq i,j)}(\lambda-\lambda_k)+(\lambda-\lambda_i)\sum_{j (\neq i)}\sum_{k (\neq i,j)}\prod_{\ell (\neq i,j,k)}(\lambda-\lambda_\ell).
\end{align*}
Substituting $\lambda=\lambda_i$,
\begin{align*}
f_\lambda(\lambda_i)=\prod_{j (\neq i)}(\lambda_i-\lambda_j),\ 
f_{\lambda \lambda}(\lambda_i)=2\sum_{j (\neq i)}\prod_{k (\neq i,j)}(\lambda_i-\lambda_k).
\end{align*}
On the other hand, the right-hand side of (A.1) is 
\begin{align*}
\sum_{j(\neq i)}\frac{1}{\lambda_i-\lambda_j}=\frac{\sum_{j (\neq i)}\prod_{k (\neq i,j)}(\lambda_i-\lambda_k)}{\prod_{j (\neq i)}(\lambda_i-\lambda_j)},
\end{align*}
and the claim holds.
\end{proof}

\begin{lemma}[Minor determinants and eigenvalues]\label{lemA2}\ \\
Assume that an $N \times N$ matrix $A$ has eigenvalues $\lambda_1,\dots,\lambda_N$. Then for $k=1,\dots,N$, 
\begin{align}\label{A2}
\sum_{1\le j_1<\cdots<j_k\le N}\ \prod_{\ell=1}^k\lambda_{j_\ell}=\sum_{1\le j_1<\cdots<j_k\le N}\underset{1\le \ell,m\le k}{{\rm det}}(A_{j_\ell j_m}),
\end{align}
where $\underset{1\le \ell,m\le k}{{\rm det}}(A_{j_\ell j_m})$ is the $k$-th principal minor indexed by $\{ j_1<\dots<j_k\} \subset \{1,\dots,N\}$: 
\begin{align*}
\underset{1\le \ell,m\le k}{{\rm det}}(A_{j_\ell j_m})=\det
\begin{pmatrix}
a_{j_1j_1} & a_{j_1j_2} & \cdots & a_{j_1j_k}\\
a_{j_2j_1} & a_{j_2j_2} & \cdots & a_{j_2j_k}\\
\vdots & & \ddots & \vdots\\
a_{j_kj_1} & a_{j_kj_2} & \cdots & a_{j_kj_k}
\end{pmatrix}.
\end{align*}
\end{lemma}
\begin{proof}
Applying the binomial expansion to the characteristic polynomial $f(\lambda)=\det(\lambda I_N-A)$, we have
\begin{align*}
f(\lambda)=\lambda^N+\sum_{k=1}^N(-1)^k\lambda^{N-k}\sum_{1\le j_1<\cdots<j_k\le N}\ \prod_{\ell=1}^k\lambda_{j_\ell}.
\end{align*}
We also apply the Fredholm determinant expansion to $f(\lambda)$ and obtain 
\begin{align*}
f(\lambda)=\lambda^N+\sum_{k=1}^{N}(-1)^k\lambda^{N-k}\sum_{1\le j_1<\cdots<j_k\le N}\underset{1\le \ell,m\le k}{{\rm det}}(A_{j_\ell j_m}).
\end{align*}
Therefore, the claim holds by comparing the coefficient of $\lambda^{N-k}$ on each other. 
\end{proof}

\begin{lemma}[Twice cofactor expansion form]\label{lemA3}\ \\ 
For any $N \times N$ matrix $A$ and fixed integers $k<\ell$, 
\begin{align}\label{A3}
\det A&=a_{kk}\det(A_{k|k})-a_{k\ell}a_{\ell k}\det(A_{k\ell |\ell k})+\sum_{\substack{q\neq k,\ell\\q<\ell}}(-1)^{k+q-1}a_{k\ell}a_{\ell q}\det(A_{k\ell | \ell q})\nonumber\\
&+\sum_{\substack{q\neq k,\ell\\q>\ell}}(-1)^{k+q}a_{k\ell}a_{\ell q}\det(A_{k\ell | \ell q})+\sum_{\substack{p\neq k,\ell\\p>k}}(-1)^{\ell+p-1}a_{kp}a_{\ell k}\det(A_{k\ell |pk})\nonumber\\
&+\sum_{\substack{p\neq k,\ell\\p<k}}(-1)^{\ell+p}a_{kp}a_{\ell k}\det(A_{k\ell |pk})+\sum_{\substack{p\neq k,\ell,\ q \neq k \\p>q}}(-1)^{k+\ell+p+q-1}a_{kp}a_{\ell q}\det(A_{k\ell |pq})\nonumber\\
&+\sum_{\substack{p\neq k,\ell,\ q \neq k \\p<q}}(-1)^{k+\ell+p+q}a_{kp}a_{\ell q}\det(A_{k\ell |pq}).
\end{align}
\end{lemma}
\begin{proof}
We expand $\det A$ by the $k$-th row, and we also expand each of the $(N-1)$-th determinants by the $\ell$-th row.
\end{proof}

\begin{lemma}[Cauchy-Binet formula, \cite{HJ}]\label{lemA4}\ \\
Let $A \in M_{m,n}(\C)$. For index sets $\alpha=\{i_1,\dots,i_p\} \subseteq \{1,\dots,m\},\ p \le m$\ and $\beta=\{j_1,\dots, j_q\} \subseteq \{1,\dots, n\},\ q \le n$,\ the $p \times q$ submatrix $A(\alpha,\beta)$ is defined as $A(\alpha,\beta)_{r,s}:=A_{i_r,j_s}$. Here, $i_1<\cdots<i_p$ and $j_1<\cdots<j_q$ hold and these index sets are ordered lexicographically. When $\sharp\alpha = \sharp\beta = k \le min\{m,n\}$, the $k$-th compound matrix of $A$ is defined as the $\binom{m}{k} \times \binom{n}{k}$ matrix whose $(\alpha, \beta)$ entry is $\det(A(\alpha,\beta))$ and we denote this by $C_k(A)$.\\
Let $A \in M_{m,k}(\C)$,\ $B \in M_{k,n}(\C)$ and $C:=AB$. We fix the index sets $\alpha \subseteq \{1,\dots,m\}$ and $\beta \subseteq \{1,\dots, n\}$ where $\sharp\alpha = \sharp\beta =r \le \min\{m,k,n\}$. Then the determinant of the submatrix $C(\alpha, \beta)$ has an expression:
\begin{align}
\det(C(\alpha, \beta))=\sum_{\gamma}\det(A(\alpha,\gamma))\det(B(\gamma,\beta)),
\end{align}
where the summation is taken over all index sets $\gamma \subseteq \{1,\dots,k\}$ of cardinality $r$.
\end{lemma}

\begin{lemma}[Sylvester's\ identity, \cite{HJ}]\label{lemA5}\ \\
For any square matrix $A$, we have
\begin{align}
|A|\times|A_{ij|kl}|=
\begin{vmatrix}
|A_{i|k}| & |A_{i|l}| \\
|A_{j|k}| & |A_{j|l}|
\end{vmatrix}. 
\end{align}
\end{lemma}

\begin{lemma}[Strong Cauchy's interlacing law, \cite{HJ}]\label{strsep}\ \\
Let $A$ be an $n \times n$ symmetric tridiagonal matrix:
\begin{align*}
A:=
\begin{pmatrix}
a_1 & b_{1}  & 0 & \cdots & \cdots & \cdots & 0\\
 b_{1} & a_2 &  b_{2} & \ddots &  &  & \vdots \\
 0 & b_{2} & a_3 & \ddots & \ddots & & \vdots\\
 \vdots & \ddots & \ddots& \ddots& \ddots&  \ddots& \vdots\\
\vdots & & \ddots & \ddots & a_{n-2} & b_{n-2} & 0\\
 \vdots &   &  & \ddots & b_{n-2} & a_{n-1} & b_{n-1}\\
 0  & \cdots &\cdots & \cdots & 0 & b_{n-1} & a_n
\end{pmatrix}.
\end{align*}
Let $\lambda_1\le \cdots \le \lambda_n$ and $\eta_1\le \cdots \le \eta_{n-1}$ be the eigenvalues of A and that of the $(n-1)$-th leading principal minor of $A$, respectively. Suppose $b_i\neq 0$ for $1 \le i \le n-1$, then the strong Cauchy's interlacing law holds:
\begin{align}\label{sci}
\lambda_1<\eta_1<\lambda_2<\cdots<\lambda_{n-1}<\eta_{n-1}<\lambda_n.
\end{align}
In particular, the eigenvalues $\{\lambda_i\}_{1\le i \le n}$ of $A$ are distinct.
\end{lemma}
\begin{proof}
For $2 \le k \le n$, let $A^{(k)}$ be the $k$-th leading principal minor of $A$, and $f_k(\lambda):=\det(\lambda I_k-A^{(k)})$. Expanding determinants, we have
\begin{align}\label{trieq}
f_{k}(\lambda)=(\lambda-a_k)f_{k-1}(\lambda) - b_{k-1}^2f_{k-2}(\lambda), \ \ 2 \le k \le n,
\end{align}
where $f_0(\lambda)\equiv1$. By the Cauchy's interlacing law \cite{HJ}, we have
$$
\lambda_1\le \eta_1\le \lambda_2 \le \cdots \le \eta_{n-1} \le\lambda_n.
$$
Hence, to show the claim, it is sufficient to prove the strong interlacement. Suppose that there exists $\lambda_0 \in \R$ such that $f_n(\lambda_0)=f_{n-1}(\lambda_0)=0$. Then, by the assumption of $\{b_i\}$ and using \eqref{trieq} for $k=n-1,\dots,2$ repeatedly, we have $f_0(\lambda_0)=0$, which is the contradiction.
\end{proof}

\section*{Acknowledgments}
The author wishes to express his thanks to Prof. Hideki Tanemura for suggesting the problem  and giving him insightful comments.

\end{document}